\documentclass[11pt,leqno]{amsart}
\usepackage{amsmath}
\usepackage{amsfonts}
\usepackage{amssymb}
\usepackage{graphicx}
\usepackage{color}
\newtheorem{theorem}{Theorem}[section]
\newtheorem{lemma}[theorem]{Lemma}

\newtheorem{proposition}[theorem]{Proposition}

\usepackage{hyperref}\usepackage{xcolor}

\newcommand{\mres}{\mathbin{\vrule height 1.6ex depth 0pt width
0.13ex\vrule height 0.13ex depth 0pt width 1.3ex}}

\def\R{\Bbb R}

\def\H{\text{\rm Hess}}

\def\oint{\frac{\ \ }{\ \ }{\hskip -0.4cm}\int}

\numberwithin{equation}{section}

\begin{document}

\title[$L^p$-Minkowski flow]{\bf Anisotropic flow, entropy and $L^p$-Minkowski problem }

\author{K\'aroly J. B\"or\"oczky}
\address{Alfr\'ed R\'enyi Institute of Mathematics\\
        Hungarian Academy of Sciences\\
       }
\email{carlos@renyi.hu}

\author{Pengfei Guan}
\address{Department of Mathematics and Statistics\\
        McGill University\\
        Montreal, Quebec, H3A 2K6, Canada.}
\email{pengfei.guan@mcgill.ca}

\thanks{}
\begin{abstract}
We provide a natural simple argument using anistropic flows to prove the existence of weak solutions to Lutwak's $L^p$-Minkowski problem on $S^n$ which were obtained by other methods.
\end{abstract}
\subjclass{35K55, 35B65, 53A05, 58G11}

\maketitle

\leftskip 0 true cm \rightskip 0 true cm
\renewcommand{\theequation}{\arabic{equation}}
\setcounter{equation}{0} \numberwithin{equation}{section}

\section{Introduction}

For $\alpha>0$ and non-negative $f\in L^1(\mathbb{S}^n)$ with positive integral, we are interested finding a weak solution to the Monge-Amp\'ere equation 
\begin{equation}\label{eq:min1}
u^{\frac1\alpha}\det(\bar{\nabla}^2_{ij} u+u\bar{g}_{ij})=f,
\end{equation}
or in other words, a weak solution to Lutwak's $L^p$-Minkowski problem on $S^n$ when $-n-1<p<1$
for $p=1-\frac1\alpha$ where
$\bar{\nabla}$ is the Levi-Civita connection of $\mathbb{S}^n$, $\bar{g}_{ij}$, with $\bar{g}$ being the induced round metric on the unit sphere. By a weak (Alexandrov)  solution we mean the following: Given a non-trivial
 finite Borel meaure $\mu$ on $\mathbb{S}^n$ (for example, $d\mu=f\,d\theta$ for the Lebesgue meaure $\theta$ on $S^n$
and the $f$ in \eqref{eq:min1}), find a convex body $\Omega\subset\R^{n+1}$ with $o\in\Omega$ such that
\begin{equation}\label{eq:Alexandrov}
d\mu=u^{\frac1\alpha}\,dS_\Omega
\end{equation}
where $u(x)=\max_{z\in\Omega}\langle x,z\rangle$ is the support function 
and $S_\Omega$ is the surface area measure of $\Omega$ (see Schneider \cite{Sch14}). If $\partial \Omega$ is $C^2_+$,
then 
$$
dS_\Omega=\det(\bar{\nabla}^2_{ij} u+u\bar{g}_{ij})d\theta=K^{-1}d\theta
$$ 
where $K(x)$ is the Gaussian curvature at the point of $\partial\Omega$ where $x\in S^n$ is the exterior unit normal
(see Schneider \cite{Sch14}). Concerning the regularity of the solution
of \eqref{eq:min1}, if 
$f\in C^{0,\beta}(S^n)$ and $u$ are positive, then $u$ is $C^{2,\beta}$ according to Caffarelli's regularity theory
in \cite{Caffarelli1,Caffarelli2}. On the other hand,  even if $f$ is positive and continuous for $\alpha>\frac1n$, there might exists 
weak solution where $u(x)=0$ for some $x\in S^n$ and $u$ is not even $C^1$ 
according to Example~4.2 in Bianchi, B\"or\"oczky, Colesanti \cite{BBC20}. Moreover,  
even if $f\in C^{0,\beta}(S^n)$ is positive, it is possible that $u(x)=0$ for some $x\in S^n$ for $\alpha>\frac1n$, but
 Kyeongsu Choi, Minhyun Kim, Taehun Lee \cite{Choi-Kim-Lee} still managed to obtain some regularity results in this case.

The case $\alpha=\frac1{n+2}$ of the Monge-Amp\'ere equation \eqref{eq:min1} is the critical case 
when the left hand side of \eqref{eq:min1} is invariant under linear transformations of $\Omega$,
and the case $\alpha=1$ is the so-called logarithmic Minkowski problem posed by Firey \cite{Firey}.
Setting $p=1-\frac1\alpha<1$, the Monge-Amp\'ere equation \eqref{eq:min1} is Lutwak's
 $L^p$-Minkowski problem 
\begin{equation}\label{eq:Mongep}
u^{1-p}\det(\bar{\nabla}^2_{ij} u+u\bar{g}_{ij})=f.
\end{equation}
In this notation, \eqref{eq:Alexandrov} reads as
\begin{equation}\label{eq:Alexandrovp}
d\mu=u^{1-p}\,dS_\Omega
\end{equation}
that equation makes sense for any $p\in\R$.
Within the rapidly developing $L_p$-Brunn-Minkowski theory (where $p=1$ is the classical case originating from Minkowski's oeuvre)
initiated by Lutwak \cite{Lut93,Lut93a,Lut96},
if $p>1$ and $p\neq n+1$, then Hug, Lutwak, Yang, Zhang \cite{HLYZ2} (improving on Chou, Wang \cite{Chou-Wang}) prove that
\eqref{eq:Alexandrovp} has an Alexandrov solution if and only if the $\mu$ is not concentrated onto any closed hemisphere,
and the solution is unique.
We note that there are examples  in \cite{GL} (see also \cite{HLYZ2} ) show that if $1<p<n+1$, then it may happen that the density function $f$ is a positive continuous in \eqref{eq:Mongep} and $o\in\partial K$ holds for the unique  Alexandrov solution,
and actually  Bianchi, B\"or\"oczky, Colesanti \cite{BBC20} exhibit an example that
$o\in\partial K$ even if
the density function $f$ is a positive continuous in \eqref{eq:Mongep} assuming $-n-1<p<1$ .

 In the case $p\in(0,1)$ (or equivalently, $\alpha>1$), if the
measure $\mu$ is not concentrated onto any great subsphere of $S^n$, then Chen, Li, Zhu \cite{CLZ17} prove that there 
exists an Alexandrov solution $K\in\mathcal{K}_o^n$ of \eqref{eq:Alexandrovp} using a variational argument
(see also Bianchi, B\"or\"oczky, Colesanti, Yang \cite{BBCY19}). We note for $p\in(0,1)$ and $n\geq 2$, no complete characterization of $L_p$ surface area measures is known (see B\"or\"oczky, Trinh \cite{BoT17} for 
 the case $n=1$, and  \cite{BBCY19} and
Saroglou \cite{Sar} for partial results about the case when $n\geq 2$ and the support of $\mu$ is contained in a great subspare of $S^n$).

Concerning the case $p=0$ 
(or equivalently, $\alpha=1$), the still open logarithmic Minkowski problem \eqref{eq:Mongep} or \eqref{eq:Alexandrovp}  
 was posed by Firey \cite{Firey} in 1974. The paper B\"or\"oczky, Lutwak, Yang, Zhang \cite{BLYZ13} characterized 
 even measures $\mu$ such that \eqref{eq:Alexandrovp} has an even solution for $p=0$
by the so-called subspace concentration condition (see (a) and (b) in Theorem~\ref{mainT1}). In general,
Chen, Li, Zhu \cite{CLZ19} proved that if a non-trival finite Borel measure $\mu$ on $S^{n-1}$  satisfies the same
subspace concentration condition, then 
 \eqref{eq:Alexandrovp} has a solution for $p=0$.
 On the other hand, B\"or\"oczky, Hegedus \cite{BoH15} provide conditions on the restriction of the $\mu$  in \eqref{eq:Alexandrovp}   to a pair of antipodal points.

If $-n-1<p<0$ (or equivalently, $\frac1{n+2}<\alpha<1$) and 
$f\in L_{\frac{n+1}{n+1+p}}(S^{n})$ in \eqref{eq:Mongep}, then \eqref{eq:Mongep} has a solution according to
Bianchi, B\"or\"oczky, Colesanti, Yang \cite{BBCY19}. 
For a rather special discrete measure $\mu$ satisfying that $\mu$ is not concentrated on any closed hemisphere and any $n$ unit vectors in the support of $\mu$ are independent, Zhu \cite{Zhu17} solves the $L_p$-Minkowski problem \eqref{eq:Alexandrovp} for $p<0$.
The $p=-n-1$  (or equivalently, $\alpha=\frac1{n+2}$) case of the $L_p$-Minkowski problem is the critical case because its link with  the ${\rm SL}(n)$ invariant
centro-affine curvature whose reciprocal is $u^{n+2}\det(\bar{\nabla}^2_{ij} u+u\bar{g}_{ij})$
 (see Hug \cite{Hug96} or Ludwig \cite{Lud10}). For positive results concerning the critical case $p=-n-1$, see for example 
Jian, Lu, Zhu \cite{JLZ16} and Li, Guang, Wang \cite{LGWc}, and for obstructions for a solution, see for example
Chou, Wang \cite{Chou-Wang} and
Du \cite{Du21}.

In the super-critical case $p<-n-1$ (or equivalently, $\alpha<\frac1{n+2}$),  there is a recent important work by Li, Guang, Wang \cite{LGWa} proving that 
for any positive $C^2$ function $f$, there exists a $C^4$ solution of  \eqref{eq:Mongep}. 
See also \cite{Du21} for non-existence examples. 

The main contribution of this paper is to provide a very natural argument based on  anisotropic flows developed by Andrews \cite{Andrews-pjm} to handle the case 
$-n-1<p<1$, or equivalently, the case $\frac1{n+2}<\alpha<\infty$. 

 {\it Entropy functional}. For any convex body $\Omega$, a fixed positive function $f$ on $\mathbb{S}^n$ and 
$\alpha \in (0, \infty)$, define
\begin{equation}\label{eq:def-e}
\mathcal{E}_{\alpha, f} (\Omega) :=
\sup_{z\in\Omega}{\mathcal E}_{\alpha, f}(\Omega,z),
\end{equation}
where
\begin{equation}\label{eq:def-e-z}
{\mathcal E}_{\alpha, f}(\Omega,z) :=
\begin{cases}
\frac{\alpha}{\alpha-1}\log\left(\oint_{\mathbb{S}^n} u_{z}(x)^{1-\frac1\alpha}\,f(x) d\theta(x)\right),&\alpha\neq 1;\\
\oint_{\mathbb{S}^n}  \log(u_{z_0}(x))\, f(x) d\theta(x),&\alpha=1.
\end{cases}
\end{equation}
 Here $u_{z_0}(x):=\sup_{z\in\Omega}\left\langle z-z_0,x\right\rangle$ is the \emph{support function} of $\Omega$ in direction $x$ with respect to $z_0$ and 
$\oint_{\mathbb{S}^n} h(x)\, d\theta(x)=\frac{1}{\omega_n} \int_{\mathbb{S}^n} h(x)$ with $\omega_n$  being the surface area of $\mathbb{S}^n$ and $\theta$ is the Lebesgue measure on $S^n$. 
When $\alpha=1$ and $f(x)\equiv 1$, then the above quantity agrees with the entropy in \cite{GN},  first introduced by Firey \cite{Firey} for the centrally symmetric $\Omega$. General integral quantities studied by Andrews in \cite{Andrews-IMRN, Andrews-pjm}.
Here we shall assume that $\oint_{\mathbb{S}^n} f(x)\, d\theta(x)=1$; namely,  $\frac1{\omega_n}\,f(x)d\theta(x)$ is a probability measure.  For the special case $f\equiv 1$, $\mathcal{E}_{\alpha, f}(\Omega) $ becomes the entropy $\mathcal{E}_\alpha(\Omega)$  in \cite{AGN16}.

\medskip

For positive $f\in C^{\infty}(\mathbb S^n)$, consider the anisotropic flow for convex hypersurfaces  $\tilde X(\cdot, \tau): M_{\tau}\to \mathbb{R}^{n+1}$:
\begin{equation}\label{eq:gcf-f-un}
\frac{\partial}{\partial \tau}\tilde X(x, \tau)= -f^{\alpha} (\nu) \tilde K^{\alpha}(x, \tau)\, \nu(x, \tau),
\end{equation}
where $\nu(x, \tau)$ is the unit exterior normal at $\tilde X(x, \tau)$ of $\tilde M_\tau=\tilde X(M, \tau)$, and $\tilde K(x,\tau)$ is the Gauss curvature of $\tilde M_\tau$ at $\tilde X(x,\tau)$.  Andrews \cite{Andrews-pjm} proved that flow (\ref{eq:gcf-f-un}) contracts to a point under finite time if the initial hypersurface $M_0$ is strictly convex. Under a proper normalization, the normalized anisotropy flow of (\ref{eq:gcf-f-un}) is  
\begin{equation}\label{gcf-f-nor}
\frac{\partial}{\partial t}X(x, t)= -\frac{f^\alpha(\nu) K^{\alpha}(x, t)}{\oint_{\mathbb{S}^n}f^\alpha  K^{\alpha-1}}\, \nu(x, t) +X(x,t).
\end{equation} 


\medskip

The basic observation is that a critical point for entropy $\mathcal{E}_{\alpha, f} (\Omega)$ defined in (\ref{eq:def-e}) under volume normalization is a solution to equation (\ref{eq:min1}). The entropy is monotone along flow (\ref{gcf-f-nor}). One may view  (\ref{eq:min1}) is an "optimal solution" to this variational problem as the flow (\ref{gcf-f-nor}) provides a natural path to reach it. This approach was  devised in \cite{ABGN}  with the aim to obtain convergence of the normalized flow (\ref{gcf-f-nor}). The main arguments in \cite{ABGN} follows those in \cite{GN, AGN16} where convergence of isotropic flows by power of Gauss curvature (i.e., $f=1$) was established. Unfortunately, the entropy point estimate in \cite{GN, AGN16} fails for general anisotropic flows except $\frac{1}{n+2}<\alpha\le\frac{1}n$ (\cite{Andrews-pjm}). 
The convergence was obtained in \cite{ABGN} assuming $M_0$ and $f$ are invariant under a subgroup $G$ of $O(n+1)$ which has no fixed point. We note that an inverse Gauss curvature flow argument was considered by Bryan, Ivaki, Scheuer \cite{BIS19}  to produce a origin-symmetric solution to (\ref{eq:min1}). 

  \medskip
  
Since we are only interested in finding a weak solution to (\ref{eq:Alexandrov}), one only needs certain "weak" convergence of the flow (\ref{gcf-f-nor}).  The key steps are to control diameter with entropy under appropriate conditions on measure $\mu=f d\theta$ on $\mathbb S^n$ and use monotonicity of entropy to produce a solution to (\ref{eq:Alexandrov}). The following is our main result.

\begin{theorem}\label{mainT1} For  $\alpha>\frac1{n+2}$
and finite non-trivial Borel measure $\mu$ on $\mathbb{S}^n$, $n\geq 1$, 
there exists a weak solution of \eqref{eq:Alexandrov} provided the following holds:
\begin{description}
\item[(i)] If $\alpha>1$ and $\mu$ is not concentrated onto any great subsphere $x^\bot\cap \mathbb{S}^n$, 
$x\in \mathbb{S}^n$.

\item[(ii)] If $\alpha=1$ and $\mu$   satisfies  that
for any linear $\ell$-subspace $L\subset \R^{n+1}$ 
with $1\leq \ell\leq n$, we have
\begin{description}
\item[(a)] $\displaystyle 
\mu(L\cap \mathbb{S}^n)\leq  \frac{\ell}{n+1}\cdot \mu(\mathbb{S}^n)$;
\item[(b)] equality  in (a) for a  linear $\ell$-subspace $L\subset \R^{n+1}$ 
with $1\leq d\leq n$ implies the existence of a complementary linear $(n+1-\ell)$-subspace $\widetilde{L}\subset \R^{n+1}$
 such that
${\rm supp}\,\mu\subset L\cup \widetilde{L}$.
\end{description} 

\item[(iii)] If $\frac1{n+2}<\alpha<1$ and $d\mu=f\,d\theta$ for non-negative 
$f\in L^{\frac{n+1}{n+2-\frac1\alpha}}( \mathbb{S}^n)$ with $\int_{\mathbb{S}^n}f>0$. 

\end{description}

\end{theorem}

Let us briefly discuss what is known about uniqueness of the solution of the $L_p$-Minkowski problem \eqref{eq:Alexandrovp}. If $p>1$ and $p\neq n$, then Hug, Lutwak, Yang, Zhang \cite{HLYZ2} proved that the Alexandrov solution of 
the $L_p$-Minkowski problem \eqref{eq:Alexandrovp} is unique. However, if $p<1$, then
the solution of the $L_p$-Minkowski problem \eqref{eq:Mongep} may not be unique even if $f$ is positive and continuous.
Examples are provided by Chen, Li, Zhu \cite{CLZ17,CLZ19} if $p\in[0,1)$, and Milman \cite{Mila} shows that for any $C\in\mathcal{K}_{(0)}$, one finds $q\in (-n,1)$ such that if $p<q$, then there exist multiple solutions to
the $L_p$-Minkowski problem \eqref{eq:Alexandrovp} with $\mu=S_{C,p}$; or in other words, there exists 
$K\in\mathcal{K}_{(0)}$ with $K\neq C$ and $S_{K,p}=S_{C,p}$. In addition, 
Jian, Lu, Wang \cite{JLW15} and Li, Liu, Lu \cite{LLL22} prove that for any $p<0$, there exists positive even $C^\infty$ function $f$ with rotational symmetry
such that the $L_p$-Minkowski problem \eqref{eq:Mongep} has multiple positive  even $C^\infty$ solutions.
We note that in the case of the centro-affine Minkowski problem $p=-n$, Li \cite{Li19} even verified the possibility of existence of infinitely many solutions without affine equivalence, and Stancu \cite{Sta22} 
related unique solution in the cases $p=0$ and $p=-n$.

The case when $f$ is a constant function in the $L_p$-Minkowski problem \eqref{eq:Mongep} has received a special attention
since Firey \cite{Firey}. When $p=-(n+1)$,  \eqref{eq:Mongep} is self-similar solution of affine curvature flow.  It's proved by Andrews that all solutions are centered ellipsoids. If $n=2, \ p=2$, the uniqueness was proved by Andrews \cite{And99}.  For general $n$ and $p>-(n+1)$, through the work of Lutwak \cite{Lut93a}, Guan-Ni \cite{GN} and Andrews, Guan, Ni \cite{AGN16},  Brendle, Choi, Daskalopoulos \cite{BCD17} finally classified that the only solutions are centered balls.
See also Crasta, Fragal\'a \cite{CrF}, Ivaki, Milman \cite{IvM} and Saroglou \cite{Sar22} for other approaches.
 Stability versions of these results have been obtained by Ivaki \cite{Iva22}, but still no stability version is known in the case $p\in[0,1)$ if we allow any solutions of \eqref{eq:Mongep} not only even ones.

Concerning recent versions of the $L_p$-Minkowski problem, see B\"or\"oczky \cite{Bor}.

The paper is structured as follows: The required diameter bounds are dicussed in Section~\ref{sec:Diameter}.
Section~\ref{sec:Entropy} verifies the main properties of the Entropy, 
Section~\ref{sec:Flow} proves our main result Theorem~\ref{MongeAmpere-with-entropy} about flows,
and finally Theorem~\ref{mainT1} is proved in Section~\ref{sec:main} via weak approximation.





\section{Entropy and diameter estimates}
\label{sec:Diameter}

For $\delta\in[0,1)$  and linear $i$-subspace $L$ of $\R^{n+1}$ with $1\leq{\rm dim}\,L\leq n$, we consider the collar
$$
\Psi(L\cap \mathbb{S}^n,\delta)=\{x\in \mathbb{S}^n:\langle x,y\rangle\leq \delta\mbox{ for }y\in L^\bot\cap \mathbb{S}^n\}.
$$
Let $B(1)\subset \R^{n+1}$ be the unit ball centered at the origin.

\begin{theorem}\label{f-diamests}
Let $\alpha>\frac1{n+2}$,  let $\oint_{\mathbb{S}^n}f=1$  for a bounded measurable function $f$ on $\mathbb{S}^n$ with 
$\inf f>0$, and let
$\Omega\subset\R^{n+1}$ be a convex body such that $|\Omega|=|B(1)|$ and
${\rm diam}\, \Omega= D$. For any $\delta,\tau\in(0,1)$, we have
\begin{description}
\item[(i)]
if $\alpha>1$, and 
$\oint_{\Psi(z^\bot\cap \mathbb{S}^n,\delta)}f\leq 1-\tau$ for any $z\in S^n$, then
$$
\exp\left(\frac{\alpha-1}{\alpha}\,\mathcal{E}_{\alpha, f} (\Omega)\right)
\geq \gamma_1 \tau\delta^{1-\frac1\alpha}D^{1-\frac1\alpha}
$$
where $\gamma_1>0$ depends on $n$ and $\alpha$;
\item[(ii)] if $\alpha=1$, and 
$$
\oint_{\Psi(L\cap \mathbb{S}^n,\delta)}f< \frac{(1-\tau)i}{n+1}
$$
for any linear $i$-subspace $L$ of $\R^{n+1}$, $i=1,\ldots,n$, then
$$
\mathcal{E}_{1, f} (\Omega)\geq\tau\log D +\log\delta-4\log(n+1);
$$

\item[(iii)] if $\frac1{n+2}<\alpha<1$, $p=1-\frac1\alpha$ (where $-n-1<p<0$),
$\tau\leq\frac12\oint_{\mathbb{S}^n}f\cdot u^{1-\frac1\alpha}$
 and
\begin{equation}
\label{eq:f-diamestsiii}
\oint_{\Psi(z^\bot\cap \mathbb{S}^n,\delta)}f^{\frac{n+1}{n+1+p}}\leq 
\tau^{\frac{n+1}{n+1+p}} 
\end{equation}
for  any $z\in S^{n-1}$, 
then 
$$
\mbox{either $D\leq 16n^2/\delta^2$, \ \ or \ }
D\leq \left(\frac12\oint_{\mathbb{S}^n}f\cdot u^{1-\frac1\alpha}\right)^{\frac2{p}}.
$$
Moreover, if $\tau\leq \frac12\exp\left(\frac{\alpha-1}{\alpha}\,\mathcal{E}_{\alpha, f} (\Omega)\right)$, then
$$
\mbox{either $D\leq 16n^2/\delta^2$, \ \ or \ }
D\leq \left(\frac12\exp\left(\frac{\alpha-1}{\alpha}\,\mathcal{E}_{\alpha, f} (\Omega)\right)\right)^{\frac2{p}}.
$$
\end{description}
On the other hand, for any  $\alpha>\frac1{n+2}$,  bounded $f$ with $\inf f>0$ and $\oint_{\mathbb{S}^n}f=1$, and
$\tau\in(0,1)$, 
 there exists $\delta\in(0,1)$ such that conditions in (i), (ii) and (iii) hold
provided $\tau\leq \frac12\exp\left(\frac{1-\alpha}{\alpha}\,\mathcal{E}_{\alpha, f} (\Omega)\right)$
for the convex body  $\Omega\subset\R^{n+1}$ in the case of (iii).
\end{theorem}
\begin{proof} Given $\alpha>\frac1{n+2}$,  bounded $f$ with $\inf f>0$ and $\oint_{\mathbb{S}^n}f=1$, and
$\tau\in(0,1)$,  the existence of suitable $\delta\in(0,1)$ follows from the fact that the Lebesgue measure is a Borel measure.

Now we assume that the conditions in (i), (ii) and (iii) hold. 
We may assume that the centroid of $\Omega$ is the origin, thus Kannan, Lov\'asz, Simonovics \cite{KLS95} yields the existence of
an $o$-symmetric ellipsoid such that
\begin{equation}
\label{centroid-ellipsoid}
E\subset\Omega\subset (n+1)E,\mbox{ and hence  }-\Omega\subset (n+1)\Omega.
\end{equation}
Let $u$ be the support function of $\Omega$, and let
$R=\max\{\|y\|:\,y\in\Omega\}\geq D/2$ and  $z_0\in \mathbb{S}^n$ such that $Rz_0\in \partial \Omega$. We observe that the definition of the entropy yields
\begin{align*}
\oint_{\mathbb{S}^n}fu^{1-\frac1\alpha}&\leq \exp\left(\frac{1-\alpha}{\alpha}\,\mathcal{E}_{\alpha, f} (\Omega)\right)
\mbox{ \ if $\alpha>1$};\\
\oint_{\mathbb{S}^n}f\log u&\leq \mathcal{E}_{0, f} (\Omega);\\
\oint_{\mathbb{S}^n}fu^{1-\frac1\alpha}&\geq \exp\left(\frac{1-\alpha}{\alpha}\,\mathcal{E}_{\alpha, f} (\Omega)\right)
\mbox{ \ if $\frac1{n+2}<\alpha<1$}
\end{align*}

\noindent {\bf Case 1} $\alpha>1$

According to the condition in (i), we may choose $\zeta\in\{+1,-1\}$ such that
$$
\oint_{\Phi}f\geq \frac{\tau}2
\mbox{ \ for }\Phi=\{x\in \mathbb{S}^n:\langle x,\zeta z_0\rangle>\delta\},
$$
and hence $\frac{R\zeta z_0}{n+1}\in \Omega$ by \eqref{centroid-ellipsoid}.
Since $u_\sigma(x)\geq \langle \frac{R\zeta z_0}{n+1},x\rangle\geq \frac{R\delta}{n+1}$ for $x\in\Phi$, we have
$$
\oint_{\mathbb{S}^n}fu^{1-\frac1\alpha}\geq\oint_{\Phi}f\left(\frac{R\delta}{n+1}\right)^{1-\frac1\alpha}\geq
\frac{\tau}2\cdot\left(\frac{D\delta}{2(n+1)}\right)^{1-\frac1\alpha}.
$$

\noindent {\bf Case 2 } $\alpha=1$

To simplify notation, we consider the Borel probability measure
 $\mu(A)=\oint_Af$ on $S^n$.
Let $e_1,\ldots,e_{n+1}\in \mathbb{S}^n$ be the principal directions associated to the ellipsoid $E$ in 
\eqref{centroid-ellipsoid}, and let $r_1,\ldots,r_{n+1}>0$
be the half axes of $E$ with $r_ie_i\in\partial E$ where we may assume that $r_1\leq\ldots\leq r_{n+1}$. In particular,
\eqref{centroid-ellipsoid} yields that 
\begin{equation}
\label{QK}
(n+1)^{n+1}\prod_{i=1}^{n+1} r_i=\frac{|(n+1)E|}{|B(1)|}\geq \frac{|\Omega|}{|B(1)|}=1.
\end{equation}

We observe that for any $v\in \mathbb{S}^n$ , there exists $e_i$ such that
$|\langle v,e_i\rangle|\geq \frac{1}{\sqrt{n+1}}> \frac{\delta}{n+1}$.
For $i=1,\ldots,n+1$, we define
$$
B_i=\left\{v\in \mathbb{S}^n:|\langle v,e_i\rangle|\geq \frac{\delta}{n+1}\mbox{ and }
|\langle v,e_j\rangle|<\frac{\delta}{n+1}\mbox{ for }j>i\right\}.
$$
In particular, $B_i\subset \Psi(L_i\cap \mathbb{S}^n,\delta)$ for $i=1,\ldots,n$ and
$L_i={\rm lin}\{e_1,\ldots,e_i\}$.

It follows that $\mathbb{S}^n$ is partitioned into the Borel sets $B_1,\ldots,B_{n+1}$, and
as $B_i\subset \Psi(L_i\cap \mathbb{S}^n,\delta)$ for $i=1,\ldots,n$, we have
\begin{eqnarray}
\label{muAi}
\mu(B_1)+\ldots+\mu(B_i)&\leq &\frac{i(1-\tau)}{n+1}\mbox{ \ for $i=1,\ldots,n$}\\
\label{muAn}
\mu(B_1)+\ldots+\mu(B_{n+1})&=&1.
\end{eqnarray}
For $\zeta=\frac{1-\tau}{n+1}$, we have $0< \zeta<\frac1{n+1}$, and define
\begin{eqnarray}
\label{betai}
\beta_i&= &\mu(B_i)-\zeta\mbox{ \ for $i=1,\ldots,n$}\\
\label{betan}
\beta_{n+1}&=&\mu(B_{n+1})-\zeta-\tau
\end{eqnarray}
where (\ref{muAi}) and (\ref{muAn}) yield
\begin{eqnarray}
\label{sumbetai}
\beta_1+\ldots+\beta_i&\leq &0\mbox{ \ for $i=1,\ldots,m-1$}\\
\label{sumbetan}
\beta_1+\ldots+\beta_{n+1}&=&0.
\end{eqnarray}
As $r_ie_i\in \Omega$, it follows from  the definition of $B_i$
 that
$u(x)\geq \langle x,r_ie_i\rangle\geq r_i\cdot\frac{\delta}{n+1}$
for $x\in B_i$, $i=1,\ldots,n+1$.
We deduce from applying (\ref{QK}), (\ref{muAn}),  (\ref{betai}), (\ref{betan}),
(\ref{sumbetai}), (\ref{sumbetan}), $r_1\leq\ldots\leq r_{n+1}$
and $\zeta<\frac1{n+1}$
that
\begin{eqnarray*}
\int_{\mathbb{S}^n}\log u\,d\mu&= &
\sum_{i=1}^{n+1}\int_{B_i}\log u\,d\mu\\
&\geq &\sum_{i=1}^{n+1}\mu(B_i)\log r_i+\sum_{i=1}^{n+1}\mu(B_i)\log \frac{\delta}{n+1}
=  \sum_{i=1}^{n+1}\mu(B_i)\log r_i+\log \frac{\delta}{n+1}\\
&=&
\sum_{i=1}^{n+1}\beta_i\log r_i+\sum_{i=1}^{n+1}\zeta \log r_i
+\tau\log r_{n+1}+\log \frac{\delta}{n+1}\\
&\geq &
\sum_{i=1}^{n+1}\beta_i\log r_i+\zeta\log \frac{1}{(n+1)^{n+1}}
+\tau\log r_{n+1}+\log \frac{\delta}{n+1}\\
&= &
(\beta_1+\ldots+\beta_{n+1})\log r_{n+1}+
\sum_{i=1}^{n}(\beta_1+\ldots+\beta_i)(\log r_i-\log r_{i+1})\\
&& -(n+1)\zeta\log (n+1)
+\tau\log r_{n+1}+\log \frac{\delta}{n+1}\\
&\geq &-\log (n+1)+\tau\log r_{n+1}+\log \frac{\delta}{n+1}.
\end{eqnarray*}
 Now $D\leq (n+1){\rm diam}\,E=2(n+1)r_{n+1}\leq (n+1)^2r_{n+1}$ and $\tau<1$, and hence
\begin{eqnarray*}
-\log (n+1)+\tau\log r_{n+1}+\log \frac{\delta}{n+1}&\geq & -\log (n+1)+\tau\log \frac{D}{(n+1)^2}
 +\log \frac{\delta}{n+1}\\
&=& \log\left(\delta D^\tau\right)-(2+2\tau)\log (n+1)\\
&\geq&  \tau\log D +\log\delta-4\log(n+1).
\end{eqnarray*}
In particular, we conclude that
$$
\mathcal{E}_{1, f} (\Omega)\geq  \oint_{\mathbb{S}^n}f\log u  =\int_{\mathbb{S}^n}\log u\,d\mu \geq\tau\log D +\log\delta-4\log(n+1).
$$

\noindent {\bf Case 3} $\frac1{n+2}<\alpha<1$

In this case, $-(n+1)<1-\frac1\alpha<0$.
We may assume that
$$
D\geq 16n^2/\delta^2,
$$
and we consider
\begin{eqnarray*}
\Phi_0&=&\left\{x\in \mathbb{S}^n:\,u(x)> \sqrt{2R}\right\};\\
\Phi_1&=&\left\{x\in \mathbb{S}^n:\,u(x)\leq  \sqrt{2R}\right\}.
\end{eqnarray*}
Concerning $\Phi_0$, we have
\begin{equation}
\label{Phi0est}
\oint_{\Phi_0}f\cdot u^{1-\frac1\alpha}\leq (2R)^{\frac12(1-\frac1\alpha)}\int_{\Phi_0}f\leq
D^{\frac12(1-\frac1\alpha)}=D^{\frac{p}2}.
\end{equation}
On the other hand, we have $\pm \frac{R}{(n+1)}\,z_0\in\Omega$ by \eqref{centroid-ellipsoid}, thus any 
$x\in\Phi_1$ satisfies
$$
\sqrt{2R}\geq u(x)\geq \left|\left\langle x,\frac{R}{n+1}\,z_0\right\rangle\right|,
$$
and hence $|\langle x,z_0\rangle|\leq (n+1)\sqrt{\frac{2}{R}}\leq \frac{4n}{\sqrt{D}\leq \delta}$; or in other words,
$$
\Phi_1\subset \Psi(z_{0}^\bot\cap \mathbb{S}^n,\delta).
$$
It follows from $|\Omega|=|B(1)|$ and the Blaschke-Santal\'o inequality ({\it cf. } Schneider \cite{Sch14}) that
$$
\int_{\mathbb{S}^n} u^{-(n+1)}\leq (n+1)|B(1)|=\omega_{n},\mbox{ \ and hence \ }
\oint_{\mathbb{S}^n} u^{-(n+1)}\leq 1.
$$
For $p=1-\frac1\alpha\in(-n-1,0)$,    H\"older's inequality and $\int_{\Phi_1}f^{\frac{n+1}{n+1+p}}<
\tau^{\frac{n+1}{n+1+p}}$  yield
$$
\oint_{\Phi_1}f\cdot u^{1-\frac1\alpha}\leq
\left(\oint_{\Phi_1}f^{\frac{n+1}{n+1+p}}\right)^{\frac{n+1+p}{n+1}}
\left(\oint_{\Phi_1} u_\sigma^{-(n+1)}\right)^{\frac{|p|}{n+1}}\leq
\left(\oint_{\Phi_1}f^{\frac{n+1}{n+1+p}}\right)^{\frac{n+1+p}{n+1}}\leq
\tau.
$$
Finally, adding the last estimate to \eqref{Phi0est} yields 
$$
\exp\left(\frac{\alpha-1}{\alpha}\,\mathcal{E}_{\alpha, f} (\Omega)\right)\leq
 \oint_{\mathbb{S}^n}f\cdot u^{1-\frac1\alpha}\leq D^{\frac{p}2}+\tau,
$$
and hence the conditions either $\tau\leq \frac12\oint_{\mathbb{S}^n}f\cdot u^{1-\frac1\alpha}$ or 
$\tau\leq \frac12\exp\left(\frac{1-\alpha}{\alpha}\,\mathcal{E}_{\alpha, f} (\Omega)\right)$
on $\tau$ implies (iii).
\end{proof}

 \section{Anisotropic flows and monotonicity of entropies}
\label{sec:Entropy}

The following Theorem was proved by Andrews in \cite{Andrews-pjm} (see also for a discussion of contracting of non-homogeneous fully nonlinear anisotropic  curvature flows in \cite{GHL}). 

\begin{theorem}\label{andrewsT}   (Theorem 13 \cite{Andrews-pjm} ) For any $\alpha>0$ and positive $f\in C^{\infty}(\mathbb S^n)$ and any initial smooth, strictly convex hypersurface $\tilde M_0\subset \mathbb R^{n+1}$, the hypersurfaces $\tilde M_{\tau}$ given by the
solution of (\ref{eq:gcf-f-un}) exist for a finite time $T$ and converge in Hausdorff distance
to a point $p \in \mathbb R^{n+1}$ as $t$ approaches $T$. \end{theorem}

Assume \[\oint_{\mathbb S^n}f=1, \quad |\Omega_0|=|B(1)|,\] solution (\ref{eq:gcf-f-un}) yields a smooth convex solution to the normalized flow (\ref{gcf-f-nor}) with volume preserved.

\medskip

Set \begin{equation}\label{h-eq} h_z(x,t)\doteqdot f(x) u_z^{-\frac{1}{\alpha}}(x,t)K(x,t), \quad d\sigma_t(x) =\frac{u_z(x,t)}{K(x,t)}d\theta(x),\end{equation} Note that that $\oint_{\mathbb{S}^n} d\sigma_t(x) =\oint_{\mathbb{S}^n} d\theta(x) =1$.


\medskip

Since the un-normalized flow (\ref{eq:gcf-f-un}) shrinks to a point in finite time, we may assume it's the origin. Then the support function $u(x,t)$ is positive for the normalized flow (\ref{gcf-f-nor}).  


\begin{lemma}\label{cov-e} \begin{enumerate} 
\item[(a).] The entropy $ \mathcal{E}_{\alpha, f}(\Omega_{t})$ defined in (\ref{eq:def-e}) is  monotonically decreasing,  \begin{equation}\label{e-mono} \mathcal{E}_{\alpha, f}(\Omega_{t_2})\le \mathcal{E}_{\alpha, f}(\Omega_{t_1}), \quad \forall t_1\le t_2 \in [0, \infty).\end{equation} 
\item[(b).] There is $D>0$ depending only on $\inf f, \sup f, \alpha, \Omega_0$ such that 
\begin{equation}\label{De}
{\rm diam}\,\Omega_t=D(t)\le D, \ \forall t\ge 0.\end{equation}
\item[(c).] $\forall t_0\in [0, \infty)$,
\begin{equation}\label{100e}   \mathcal{E}_{\alpha, f}(\Omega_{t_0}, 0)\ge \mathcal{E}_{\alpha, f, \infty} +\int_{t_0}^{\infty}\left(\frac{\oint_{\mathbb{S}^n} h^{\alpha+1}(x,t)\, d\sigma_t} {\oint_{\mathbb{S}^n} h(x,t) \, d\sigma_t \cdot \oint_{\mathbb{S}^n} h^{\alpha}(x,t)\, d\sigma_t}-1\right)\, dt,\end{equation} where  \[h(x,t)=h_0(x, t), \ \mathcal{E}_{\alpha, f, \infty}\doteqdot\lim_{t\to \infty}  \mathcal{E}_{\alpha, f}(\Omega_{t}).\]
\end{enumerate}
\end{lemma}

 \begin{proof}\begin{enumerate}\item[(a).]  We follow argument in  \cite{GN}. For each $T_0>$ fixed, pick $T > T_0$. Let $a^{T}=(a^{T}_1,\cdots, a^{T}_{n+1})$ be
an interior point of $\Omega_T$. Set $u^T=u- e^{t-T}\sum_{i=1}^{n+1} a^T_ix_i$, it satisfies equation 
\begin{equation}\label{aT}
\frac{\partial}{\partial t}u^T(x, t)= -\frac{f^\alpha(x) K^{\alpha}(x, t)}{\oint_{\mathbb{S}^n}f^\alpha  K^{\alpha-1}} +u^T(x,t).\end{equation}

Note that since $a^T$ is an interior point of $\Omega_T$ and $u(x,T)$ is the support function of $\Omega_T$ with respect to $a^T$, 
 $u^T(x, T)> 0, \forall x\in \mathbb S^n$.
 We claim \[u^T(x, t)>0, \ \forall  t\in [0, T).\] Suppose $u^T(x_0, t')\le 0$ for some $0<t'<T, x_0\in \mathbb S^n$, the equation (\ref{aT}) implies $u^T(x_0, t)<0$ for
all $t>t'$, which contradicts to $u^T(x, T)> 0$. 

Set $a^T(t)=e^{t-T}a^T$. By the claim, $a^T(t)$ is in the interior of $\Omega_t, \ \forall t\le T$.  Denote \[d\sigma_{T,t}=u^T(x,t)K^{-1}(x,t)d\theta,\] we rewrite (\ref{aT}) as
\begin{equation}\label{aTT}
\frac{\partial}{\partial t}u_{a^T(t)}(x,t)= -\frac{f^\alpha(x) K^{\alpha}(x, t)}{\oint_{\mathbb{S}^n}h_{a^T(t)}^{\alpha}(x,t)\, d\sigma_{T,t}} +u_{a^T(t)}(x,t).\end{equation} We have
\[\frac{\partial}{\partial t} \mathcal{E}_{\alpha, f}(\Omega_{t}, a^T(t))=\frac{-\oint_{\mathbb{S}^n} h_{a^T(t)}^{\alpha+1}(x,t)\, d\sigma_{T,t}} {\oint_{\mathbb{S}^n} h_{a^T(t)}(x,t) \, d\sigma_{T,t} \cdot \oint_{\mathbb{S}^n} h_{a^T(t)}^{\alpha}(x,t)\, d\sigma_{T,t}}+1.\]

Thus, $\forall t<T$,
\begin{eqnarray}\label{n-mono} &\mathcal{E}_{\alpha, f}(\Omega_{t}, a^T(t))-\mathcal{E}_{\alpha, f}(\Omega_T, a^T)\\ \nonumber &= \int_{t}^{T}\oint_{\mathbb S^n} \left(\frac{\oint_{\mathbb{S}^n} h_{a^T(t)}^{\alpha+1}(x,t)\, d\sigma_{T,t}} {\oint_{\mathbb{S}^n} h_{a^T(t)}(x,t) \, d\sigma_{T,t} \cdot \oint_{\mathbb{S}^n} h_{a^T(t)}^{\alpha}(x,t)\, d\sigma_{T,t}}-1\right)\, dt\le 0.\end{eqnarray}
Therefore, \[ \mathcal{E}_{\alpha, f}(\Omega_{t})\ge\mathcal{E}_{\alpha, f}(\Omega_T, a^T), \ \forall t<T.\]
Since $a^T$ is arbitrary, (\ref{e-mono}) is proved.

\medskip

\item[(b).] The boundedness of $D(t)$ follows from Theorem~\ref{f-diamests} combined with the estimate
$\mathcal{E}_{\alpha, 1}(\Omega_{t})\leq \mathcal{E}_{\alpha, 1}(B(1))$ from (a)
  (see also \cite{AGN16, GN}).
The only non-trivial case is when $\frac1{n+2}<\alpha<1$ because we have to choose a $\tau$ independent of $t$.
However, we may choose any $\tau\in(0,1)$ with
$\tau\leq  \frac12\exp\left(\frac{1-\alpha}{\alpha}\,\mathcal{E}_{\alpha, f} (B(1))\right)$ according to 
$\mathcal{E}_{\alpha, 1}(\Omega_{t})\leq \mathcal{E}_{\alpha, 1}(B(1))$.

\medskip

\item[(c).] $\forall \epsilon>0, \ \forall t_0$ fixed, pick $T>T_0>t_0$. As $ \mathcal{E}_{\alpha, f}(\Omega_{T})$ is bounded by (a), $\exists a^T$ inside $\Omega_T$ such that $ \mathcal{E}_{\alpha, f}(\Omega_{T})\le  \mathcal{E}_{\alpha, f}(\Omega_{T}, a^T)+\epsilon$. By (\ref{n-mono}), 
\begin{eqnarray*} &\mathcal{E}_{\alpha, f}(\Omega_{t_0}, a^T(t_0))-\mathcal{E}_{\alpha, f}(\Omega_{T})\\ &\ge \int_{t_0}^{T_0}\oint_{\mathbb S^n} \left(\frac{\oint_{\mathbb{S}^n} h_{a^T(t)}^{\alpha+1}(x,t)\, d\sigma_{T,t}} {\oint_{\mathbb{S}^n} h_{a^T(t)}(x,t) \, d\sigma_{T,t} \cdot \oint_{\mathbb{S}^n} h_{a^T(t)}^{\alpha}(x,t)\, d\sigma_{T,t}}-1\right)\, dt-\epsilon.\end{eqnarray*} As $|a^T|\le D, \ \forall T$, let $T\to \infty$, \[a^T(t)\to 0, \ u^T(x,t)\to u(x,t), \ \ \mbox{ uniformly for $0\le t\le T_0, x\in \mathbb S^n$.} \] We obtain $\forall t_0<T_0$, 
\begin{eqnarray*} \mathcal{E}_{\alpha, f}(\Omega_{t_0}, 0)-\mathcal{E}_{\alpha, f, \infty}\ge \int_{t_0}^{T_0}\oint_{\mathbb S^n} \left(\frac{\oint_{\mathbb{S}^n} h^{\alpha+1}(x,t)\, d\sigma_t} {\oint_{\mathbb{S}^n} h(x,t) \, d\sigma_t \cdot \oint_{\mathbb{S}^n} h^{\alpha}(x,t)\, d\sigma_t}-1\right)\, dt-\epsilon.\end{eqnarray*}
Then let $T_0\to \infty$, 
as $\epsilon>0$ is arbitrary, we obtain (\ref{100e}). \end{enumerate}
 \end{proof}
 
 \medskip
 


\section{Weak convergence}
\label{sec:Flow}

The goal of this section is to prove the following statement.

\begin{theorem}
\label{MongeAmpere-with-entropy}
For a $C^\infty$ function $f:\mathbb{S}^n\to (0,\infty)$ and $\alpha>\frac1{n+2}$ with
$ \oint_{\mathbb{S}^n}f=1$, there exist $\lambda>0$ and a convex body 
$\Omega\subset\R^{n+1}$ with $o\in\Omega$ whose support function $u$ is a (possibly weak) solution
of the Monge-Amp\`ere equation
\begin{equation}\label{eq:MongeAmpere1}
u^{\frac1\alpha}\det(\bar{\nabla}^2_{ij} u+u\bar{g}_{ij})= f 
\end{equation}
and $\Omega$ satisfies that
\begin{equation}\label{eq:entropy-for-MongeAmpere}
\mathcal{E}_{\alpha, f} (\lambda\Omega)\leq \mathcal{E}_{\alpha, f} (B(1)), \quad |\lambda\Omega|=|B(1)|,
\end{equation}
where $C^{-1}<\lambda<C$  for a $C>1$ depending only on the $\alpha, \tau, \delta$ in Theorem \ref{f-diamests} such that $f$ satisfies the conditions in Theorem \ref{f-diamests}.
\end{theorem}

From now on, we will assume that the $f$ in Theorem~\ref{MongeAmpere-with-entropy} satisfies the corresponding condition in Theorem \ref{f-diamests} and $\Omega_0=B(1)$ in (\ref{gcf-f-nor}). We note that for any $z\in B(1)$, $v_z\leq 2$ for the support function $v_z$ of $B(1)$ at $z$, and hence
if $\alpha>\frac1{n+2}$, then
\begin{equation}
\label{Bentropyalphauppbound}
\mathcal{E}_{\alpha, f_k} (B(1))\leq\left\{
\begin{array}{rl}
\frac{\alpha}{\alpha-1}\cdot\log 2^{1-\frac1\alpha}&\mbox{ if }\alpha\neq 1;\\
\log 2&\mbox{ if }\alpha=1.
\end{array} \right.
\end{equation}

The following is a consequence of Theorem \ref{f-diamests} and Lemma \ref{cov-e}.

\begin{lemma}\label{f-entropybd} 
There exist $C_{\alpha, \tau, \delta}>0, D_{\alpha, \tau, \delta}>0$ and  $c_{\alpha, \tau, \delta}\in \mathbb R$ depending only on constants $\alpha, \tau, \delta$ in Theorem \ref{f-diamests} such that,  along (\ref{gcf-f-nor}), we have
\begin{equation}\label{cc100e}  D(t)\le D_{\alpha, \tau, \delta},\ \mathcal{E}_{\alpha, f}(\Omega_{t}, 0)\ge c_{\alpha, \tau, \delta}, \   \frac{1}{C_{\alpha, \tau, \delta}}\le \oint_{\mathbb S^n} h(x,t)d\sigma_t\le C_{\alpha, \tau, \delta}.\end{equation}\end{lemma}
\begin{proof} For each $\alpha>\frac1{n+2}$ fixed with condition on $f$ as in Theorem \ref{f-diamests},  $\mathcal{E}_{\alpha, f}(\Omega_{t})$ is bounded from below in terms of the diameter $D(t)$. Since $|\Omega_t|=|B(1)|$, we have $D(t)\ge 2$ by the Isodiametric Inequality ({\it cf.} Schneider \cite{Sch14}). By Theorem \ref{f-diamests},  $\mathcal{E}_{\alpha, f}(\Omega_{t})$ is bounded from below by a constant $c_{\alpha, \tau, \delta}>0$, and hence $\mathcal{E}_{\alpha, f, \infty} \ge c_{\alpha, \tau, \delta}$. 
It follows from Lemma \ref{cov-e} that $\mathcal{E}_{\alpha, f}(\Omega_{t})\le \mathcal{E}_{\alpha, f}(B(1))$, which estimate combined with \eqref{Bentropyalphauppbound} and Theorem~\ref{f-diamests} 
 yields $D(t)\le D_{\alpha, \tau, \delta}$ by a constant $D_{\alpha, \tau, \delta}$ depending only on constants in condition on $f$ in Theorem~\ref{f-diamests}.  Finally the inequalities follows from Lemma \ref{cov-e}.\end{proof}

Set 
\begin{equation}\label{eta} 
\eta(t)=\oint_{\mathbb{S}^n} h(x,t) \, d\sigma_{t}.
\end{equation}
We note that $\oint_{\mathbb{S}^n} h(x,t) \, d\sigma_{t} $ is monotone and bounded from below and above by Lemma \ref{f-entropybd}, and hence we have
\begin{equation}\label{etabd}C_{\alpha, \tau, \delta}\ge \lim_{t\to \infty} \oint_{\mathbb{S}^n} h(x,t)=\eta\ge \frac1{C_{\alpha, \tau, \delta}}.\end{equation}

\medskip

By Lemma \ref{cov-e} and Corollary \ref{f-entropybd}, 
\begin{equation}\label{100e1} \int_0^{\infty} \left(\frac{\oint_{\mathbb{S}^n} h^{\alpha+1}(x,t)\, d\sigma_t} {\oint_{\mathbb{S}^n} h(x,t) \, d\sigma_t \cdot \oint_{\mathbb{S}^n} h^{\alpha}(x,t)\, d\sigma_t}-1\right)\, dt<\infty.\end{equation}
Since the integrand is non-negative, $\exists t_k\to \infty$ such that
\begin{equation}\label{100e2} \lim_{k\to \infty} \left( \frac{\oint_{\mathbb{S}^n}h^{\alpha+1}(x,t_k)\, d\sigma_{t_k}} {\oint_{\mathbb{S}^n} h(x,t_k) \, d\sigma_{t_k} \cdot \oint_{\mathbb{S}^n} h^{\alpha}(x,t_k)\, d\sigma_{t_k}}-1\right)=0.\end{equation}
This implies 
\begin{equation}\label{100e3} \lim_{k\to \infty}\frac{\left(\oint_{\mathbb{S}^n}h^{\alpha+1}(x,t_k)\, d\sigma_{t_k}\right)^{\frac{1}{1+\alpha}}}{\oint_{\mathbb{S}^n} h(x,t_k) \, d\sigma_{t_k} }= \lim_{k\to \infty}\frac{\left(\oint_{\mathbb{S}^n}h^{\alpha+1}(x,t_k)\, d\sigma_{t_k}\right)^{\frac{\alpha}{1+\alpha}}}{\oint_{\mathbb{S}^n} h^{\alpha}(x,t_k)\, d\sigma_{t_k}}= 1.\end{equation}

After considering a subsequence, we may assume that \begin{equation}\label{100e5} \Omega_{t_k}\to \Omega, \quad u(x,t_k)\to u(x),\end{equation} where $u$ is the support function of $\Omega$. In view of (\ref{100e3}) and (\ref{etabd}), 
\begin{eqnarray}\label{100e4} 
 \lim_{k\to \infty} \oint_{\mathbb{S}^n}h^{\alpha+1}(x,t_k)\, d\sigma_{t_k}=\eta^{1+\alpha},\ \lim_{k\to \infty}\oint_{\mathbb{S}^n} h^{\alpha}(x,t_k)\, d\sigma_{t_k}= \eta^{\alpha}.\end{eqnarray}
 
 \medskip
 
 The following Lemma is crucial for the weak convergence, which is a refined form of the classical H\"odler inequality. 
 
 \begin{lemma}\label{holder room} Let $p,\ q\in \mathbb R^+$ with $\frac1p+\frac 1q=1$, set $\beta=\min\{\frac1p, \frac1q\}$. Let $(M,\mu)$ be a measurable space,  $\forall F\in L^p, \ G\in L^q$, 
 \begin{equation} \int_M |FG| d\mu \le \|F\|_{L^p}\|G\|_{L^q}\left(1-\beta\int_M \left(\frac{|F|^{\frac{p}2}}{(\int_M |F|^p d\mu)^{\frac12}}-\frac{|G|^{\frac{q}2}}{(\int_M |G|^qd\mu )^{\frac12}}\right)^2\right).\end{equation}\end{lemma}
 \begin{proof} We first prove the following {\bf Claim}. $\forall s, t\in \mathbb R$,
 \begin{equation}\label{cst} e^{\frac{s}{p}+\frac{t}{q}}\le \frac{e^s}{p}+\frac{e^t}{q}-\beta(e^{\frac{s}2}-e^{\frac{t}2})^2.\end{equation}
 We may assume $t\ge s$, set $\tau=t-s$, (\ref{cst}) is equivalent to
  \begin{equation}\label{cstau} e^{\frac{\tau}{q}}\le \frac{1}{p}+\frac{e^{\tau}}{q}-\beta(1-e^{\frac{\tau}2})^2, \ \forall \tau\ge 0.\end{equation}
Set \[\xi(\tau)=\frac{1}{p}+\frac{e^{\tau}}{q}-\beta(1-e^{\frac{\tau}2})^2-e^{\frac{\tau}{q}}.\]
We have $\xi(0)=0$,
\[\xi'(\tau)=\frac{e^{\frac{\tau}q}}{q}\rho, \ \mbox{where} \ \rho(\tau)=e^{\frac{\tau}{p}}(1-\beta q)+q\beta e^{\frac{\tau}2-\frac{\tau}q}-1.\]
 If $\beta=\frac1q$, then $\frac1q\le \frac12$, since $\tau\ge 0$, 
 \[\rho(\tau)=e^{\frac{\tau}{p}}(1-\beta q)+q\beta e^{\frac{\tau}2-\frac{\tau}q}-1=e^{\frac{\tau}2-\frac{\tau}q}-1\ge 0.\]
 If $\beta=\frac1p$, then $\frac1q\ge \frac12$, we have
 \begin{eqnarray*}\rho'(\tau)&=&e^{\frac{\tau}{p}}\left(\frac{1-\beta q}p+\beta q(\frac12-\frac1q)e^{\frac{\tau}2-\frac{\tau}q}\right)\\
& \ge & e^{\frac{\tau}{p}}\left(\frac{1-\beta q}p+\beta q(\frac12-\frac1q)\right)\\
&\ge & e^{\frac{\tau}{p}}\beta q(\frac12-\frac1p)\ge 0.\end{eqnarray*}
We conclude that \[\rho(\tau)\ge 0, \ \forall \tau\ge 0.\]
In turn, \[\xi'(\tau)\ge 0, \ \forall \tau\ge o.\] This yields (\ref{cstau}) and (\ref{cst}). The {\bf Claim} is verified. 

Back to the proof of the lemma. We may assume \[F\ge 0, \ g\ge 0, \ \int F^p>0, \ \int G^q>0.\] Set
\[e^s=\frac{F^p}{\int F^p}, \quad e^t=\frac{G^q}{\int G^q}.\]
Put them into (\ref{cst}) and integrate, as $\frac1p+\frac1q=1$,
\[\frac{\int FG}{(\int F^p)^{\frac1{p}}(\int G^q)^{\frac1{q}}}\le \left(1-\beta\int (\frac{F^{\frac{p}2}}{(\int F^p)^{\frac12}}-\frac{G^{\frac{q}2}}{(\int G^q)^{\frac12}})^2\right).\]
\end{proof}

\medskip

We prove weak convergence.

\begin{proposition}\label{k-weak-conv} $\forall \alpha>\frac{1}{n+2}$, suppose that  (\ref{100e5})  and (\ref{100e4}) hold. Denote \[u_{k}=u(x, t_k), \ \sigma_{n,k}=\sigma_n(u_{ij}(x,t_k)+u(x,t_k)\delta_{ij}).\] Then
\begin{equation}\label{100e6} \lim_{k\to \infty} \oint_{\mathbb S^n} |u_k^{\frac{1}\alpha}\sigma_{n,k}-\frac{f}{\eta}| d\theta =0,\end{equation}
where $\eta$ is defined in (\ref{eta}) which is bounded from below and above in (\ref{etabd}). As a consequence, there is convex body $\Omega\subset \mathbb R^{n+1}$ with $o\in \Omega$, 
\[|\Omega|=|B(1)|, \quad \mathcal{E}_{\alpha, f}(\Omega_{t})\le \mathcal{E}_{\alpha, f}(B(1)),\]
and its support function $u$ satisfies 
\begin{equation}\label{etaMongeAmpere} 
u^{\frac1{\alpha}} S_{\Omega}=\frac1{\eta} fd \theta.
\end{equation}
\end{proposition}
\begin{proof} We only need to verify (\ref{100e6}). By (\ref{100e4}), it is equivalent to prove
\begin{equation}\label{100e66} \lim_{k\to \infty} \oint_{\mathbb S^n} |u_k^{\frac{1}\alpha}\sigma_{n,k}-f\eta^{-1}(t_k)| d\theta =0.\end{equation}

Since $D(t_k)$ is bounded,
\[\oint_{\mathbb S^n} u_k^{\frac{1}{\alpha^2}}\sigma_{n,k} d\theta\le (D(t_k))^{\frac1{\alpha^2}}\oint_{\mathbb S^n} u_k^{\frac{1}{\alpha^2}}\sigma_{n,k} d\theta\le (D(t_k))^{\frac1{\alpha^2}}|\partial \Omega_{t_k}|\le C.\]
\begin{eqnarray} \oint_{\mathbb S^n} |u_k^{\frac{1}\alpha}\sigma_{n,k}-f\eta^{-1}(t_k)| d\theta &=&\oint_{\mathbb S^n} |\frac{f}{\eta(t_k)u_k^{\frac{1}\alpha}\sigma_{n,k}}-1 |u_k^{\frac{1}\alpha}\sigma_{n,k} d\theta\nonumber\\
&\le & \left(\oint_{\mathbb S^n} |\frac{f}{\eta(t_k)u_k^{\frac{1}\alpha}\sigma_{n,k}}-1 |^{1+\alpha} d\sigma_{t_k}\right)^{\frac{1}{1+\alpha}}\left(\oint_{\mathbb S^n} u_k^{\frac{1}{\alpha^2}}\sigma_{n,k} d\theta\right)^{\frac{\alpha}{1+\alpha}}\nonumber \\
&\le & C \left(\oint_{\mathbb S^n} |f\eta^{-1}(t_k)u_k^{-\frac{1}\alpha}\sigma^{-1}_{n,k}-1 |^{1+\alpha} d\sigma_{t_k}\right)^{\frac{1}{1+\alpha}}.
\end{eqnarray}

By  (\ref{100e2}), (\ref{100e4}) and Lemma \ref{holder room}, with $p=\alpha+1$, $F^{\frac{1}{1+\alpha}}=h(x,t_k)$, $G=1$
\begin{equation}\label{100e7} \lim_{k\to \infty} \oint \left((\frac{h(x,t_k)}{\eta(t_k)})^{\frac{1+\alpha}2}-1\right)^2d\sigma_{t_k}=0.\end{equation}

For $t_k$ fixed, let \[\gamma_k(x)=f\eta^{-1}(t_k)u_k^{-\frac{1}\alpha}\sigma^{-1}_{n,k}=h(x,t_k)\eta^{-1}(t_k)\] and  set \[\Sigma_{k}=\{x\in \mathbb S^n \ | \ |\gamma_k(x)-1|\le \frac12.\}\] It is straightforward to check that $\exists A_{\alpha}\ge 1$ depending only on 
$\alpha$ such that \begin{eqnarray*} A_{\alpha}|\gamma^{\frac{1+\alpha}2}_k(x)-1|&\ge& |\gamma_k(x)-1|, \ \forall x\in \Sigma_k; \\
A_{\alpha} |\gamma^{\frac{1+\alpha}2}_k(x)-1|^{2}&\ge & |\gamma_k(x)-1|^{1+\alpha}, \ \forall x\in \Sigma_{k}^c.\end{eqnarray*}
Since $ |\gamma^{\frac{1+\alpha}2}_k(x)-1|\le 2^{1+\alpha}, \ \forall x\in \Sigma_k$, let $\delta=\min\{1+\alpha, 2\}$,
\begin{eqnarray*}
\ \oint_{\mathbb S^n} |\gamma_k(x)-1 |^{1+\alpha} d\sigma_{t_k}=&\frac{1}{\omega_n}\left(\int_{\Sigma_k} |\gamma_k(x)-1 |^{1+\alpha} d\sigma_{t_k}+\int_{\Sigma_k^c} |\gamma_k(x)-1 |^{1+\alpha} d\sigma_{t_k}\right)\\
\le & \frac{A^{1+\alpha}_{\alpha}}{\omega_n} \left(\int_{\Sigma_k} |\gamma_k^{\frac{1+\alpha}2}(x)-1 |^{1+\alpha} d\sigma_{t_k}+\int_{\Sigma_k^c} |\gamma_k^{\frac{1+\alpha}2}(x)-1 |^{2} d\sigma_{t_k}\right)\\
\le & \frac{(2A_{\alpha})^{1+\alpha}}{\omega_n}  \left(\int_{\Sigma_k} |\gamma_k^{\frac{1+\alpha}2}(x)-1 |^{\delta} d\sigma_{t_k}+\int_{\Sigma_k^c} |\gamma_k^{\frac{1+\alpha}2}(x)-1 |^{2} d\sigma_{t_k}\right)\\
\le & (2A_{\alpha})^{1+\alpha} \left(\oint_{\mathbb S^n} |\gamma_k^{\frac{1+\alpha}2}(x)-1 |^{\delta} d\sigma_{t_k}+\oint_{\mathbb S^n} |\gamma_k^{\frac{1+\alpha}2}(x)-1 |^{2} d\sigma_{t_k}\right)\\
\le & (2A_{\alpha})^{1+\alpha} \left((\oint_{\mathbb S^n} |\gamma_k^{\frac{1+\alpha}2}(x)-1 |^{2} d\sigma_{t_k})^{\frac{\delta}2}+\oint_{\mathbb S^n} |\gamma_k^{\frac{1+\alpha}2}(x)-1 |^{2} d\sigma_{t_k}\right).\end{eqnarray*} 
By (\ref{100e7}),  \[\lim_{k\to \infty}\oint_{\mathbb S^n} |\gamma_k^{\frac{1+\alpha}2}(x)-1 |^{2} d\sigma_{t_k}=0.\]
It follows (\ref{100e66}).
\end{proof}
\medskip

\begin{proof}[Proof of Theorem \ref{MongeAmpere-with-entropy}] It follows from Proposition \ref{k-weak-conv} after a proper rescaling as $\eta$ satisfies (\ref{etabd}) and \eqref{etaMongeAmpere}.
\end{proof}

\medskip

\section{The general Monge-Amp\`ere equations - proof of Theorem~\ref{mainT1}}
\label{sec:main}

In order to prove Theorem~\ref{mainT1}, we need weak approximation in the following sense:

\begin{lemma}\label{weak-approximation} 
For $\delta,\varepsilon\in(0,\frac12)$ and  a Borel probability measure $\mu$ on $\mathbb{S}^n$, $n\geq 1$, 
there exists a seguence $d\mu_k=\frac1{\omega_n}\,f_k\,d\theta$ of Borel  probability measures whose weak limit is $\mu$ and
 $f_k\in C^\infty( \mathbb{S}^n)$ satisfies $f_k>0$ and the following properties:
\begin{description}
\item[(i)] If $\mu\left(\Psi(z^\bot\cap \mathbb{S}^n,2\delta)\right)\leq 1-\varepsilon$ for  any $z\in S^{n-1}$,  then
\begin{equation}
\label{eq:f-diamestsi00}
\oint_{\Psi(z^\bot\cap \mathbb{S}^n,\delta)}f_k\leq 1-\varepsilon \mbox{ \ for  any $z\in S^{n-1}$}.
\end{equation}
\item[(ii)] If $\mu(\Psi(L\cap \mathbb{S}^n,2\delta))<(1-2\delta)\cdot \frac{\ell}{n+1}$
for any linear $\ell$-subspace $L$ of $\R^{n+1}$, $\ell=1,\ldots,n$, then
\begin{equation}
\label{eq:strict-subspace-function0}
\mu_k\left(\Psi\left(L\cap \mathbb{S}^n,\delta\right)\right)<(1-\delta)\cdot \frac{\ell}{n+1}.
\end{equation}

\item[(iii)] If $d\mu=\frac1{\omega_n}\,f\,d\theta$ for 
$f\in L^{r}(\mathbb{S}^n)$ where $r>1$,
and
\begin{equation}
\label{eq:f-diamestsiii0}
\oint_{\Psi(z^\bot\cap \mathbb{S}^n,2\delta)}f^r\leq \varepsilon
\end{equation}
for  any $z\in S^{n-1}$,  then
\begin{equation}
\label{eq:f-diamestsiii00}
\int_{\Psi(z^\bot\cap \mathbb{S}^n,\delta)}f_k^r\leq 2^r\varepsilon \mbox{ \ for  any $z\in S^{n-1}$}.
\end{equation}

\end{description}

\end{lemma}
\begin{proof}
 For $k\geq 1$,
let $\{B_{k,i}\}_{i=1,\ldots,m(k)}$ be a partition of $S^n$ into spherically convex Borel measurable sets
$B_{k,i}$ with ${\rm diam}B_{k,i}\leq \frac1k$ and $\theta(B_{k,i})>0$.
For each $B_{k,i}$, we choose a $C^\infty$ function $h_{k,i}:\mathbb{S}^n\to[0,\infty)$ such that
for $M_{k,i}=\max h_{k,i}$ and the probbility measure $d\tilde{\theta}=\frac1{\omega_n}\,d\theta$, we have
\begin{itemize}
\item $h_{k,i}=0$ if $x\not\in B_{k,i}$;
\item $M_{k,i}\leq (1+\frac1k)\cdot \frac{\mu(B_{k,i})}{\tilde{\theta}(B_{k,i})}$;
\item $\theta\left(\left\{x\in B_{k,i}:h_{k,i}(x)<M_{k,i}\right\}\right)<\frac1k\,\theta(B_{k,i})$;
\item $\int_{B_{k,i}}h_{k,i}\,d\tilde{\theta}=\mu(B_{k,i})$.
\end{itemize} 
We consider the positive $C^\infty$ function $\tilde{f}_k=\frac1k+\sum_{i=1}^{m(k)}h_{k,i}$, and hence 
$f_k=\left(\oint_{\mathbb{S}^n}\tilde{f}_k\right)^{-1}\tilde{f}$ satisfies that the probability measure
$d\mu_k=f_k\,d\tilde{\theta}$ tends weakly to $\mu$, and for large $k\geq 1/\delta$, $\mu_k$ satisfies (i), and if (ii) holds then 
$\mu_k$ also satisfies
\eqref{eq:strict-subspace-function0}.

Turning to (iii), we assume that $d\mu=f\,d\tilde{\theta}$ for $f\in L^{r}(\mathbb{S}^n)$ where $r>1$, 
and $f$ satisfies \eqref{eq:f-diamestsiii0}. For any large $k$ and $i=1,\ldots,m(k)$, we deduce from the H\"older inequality that
\begin{align*}
\oint_{B_{k,i}} \tilde{f}_k^r&=\oint_{B_{k,i}} \left(h_{k,i}+\frac1k\right)^r\leq
2^{r-1}\oint_{B_{k,i}} h_{k,i}^r+2^{r-1}\oint_{B_{k,i}} \frac1{k^r}\\
&\leq
2^{r-1}\tilde{\theta}(B_{k,i})M_{k,i}^r+2^{r-1}\oint_{B_{k,i}} \frac1{k^r}\\
&\leq 2^{r-1}\left(1+\frac1k\right)^r\tilde{\theta}(B_{k,i})
\left(\frac{\int_{B_{k,i}} f}{\tilde{\theta}(B_{k,i})}\right)^r+2^{r-1}\oint_{B_{k,i}} \frac1{k^r}\\
&\leq 2^{r-1}\left(1+\frac1k\right)^r\oint_{B_{k,i}} f^r+2^{r-1}\oint_{B_{k,i}}\frac1{k^r}.
\end{align*}
Summing this estimate up for large $k$ and all $B_{k,i}$ with 
$B_{k,i}\cap \Psi(z^\bot\cap \mathbb{S}^n,\delta)\neq \emptyset$, and using that
$\oint_{\mathbb{S}^n}\tilde{f}_k\geq 2^{-1/2}$ for large $k$, we deduce that
$$
\oint_{\Psi(z^\bot\cap \mathbb{S}^n,\delta)} f_k^r\leq \sqrt{2}
\oint_{\Psi(z^\bot\cap \mathbb{S}^n,\delta)} \tilde{f}_k^r\leq
\sqrt{2}\cdot2^{r-1}\left(1+\frac1k\right)^r\oint_{\Psi(z^\bot\cap \mathbb{S}^n,2\delta)} f^r+
\sqrt{2}\cdot\frac{2^{r-1}}{k^r}\leq
2^r\varepsilon.
$$
\end{proof}

For $\alpha>0$ and $p=1-\frac1\alpha$, the $L_p$-surface area $dS_{\Omega,p}=u^{1-p}dS_\Omega$ was introduced in the seminal works  Lutwak \cite{Lut93,Lut93a,Lut96} for a convex body $\Omega\subset\R^{n+1}$
with $o\in\Omega$ and support function $u$. Since the surface area measure is weakly continuous for $p<1$, and if $K\subset\R^{n+1}$
is an at most $n$ dimensional compact convex set, then $S_{K,p}\equiv 0$ for $p<1$, we have the following statement.

\begin{lemma}
\label{non-degenerate-weak-limit}
If convex bodies $\Omega_m\subset\R^{n+1}$ tend to a compact convex set $K\subset\R^{n+1}$ where $o\in\Omega_m,K$,
and $\liminf_{m\to \infty}S_{\Omega_m,p}>0$, then ${\rm int}K\neq \emptyset$ and
$S_{\Omega_m,p}$ tends weakly  to $S_{K,p}$.
\end{lemma}

For the reader's sake, let us recall Theorem~\ref{mainT1}:

\begin{theorem}\label{mainT10} For  $\alpha>\frac1{n+2}$
and finite non-trivial Borel measure $\mu$ on $\mathbb{S}^n$, $n\geq 1$, 
there exists a weak solution of \eqref{eq:Alexandrov} provided the following holds:
\begin{description}
\item[(i)] If $\alpha>1$ and $\mu$ is not concentrated onto any great subsphere $x^\bot\cap \mathbb{S}^n$, 
$x\in \mathbb{S}^n$.

\item[(ii)] If $\alpha=1$ and $\mu$   satisfies  that
for any linear $\ell$-subspace $L\subset \R^{n+1}$ 
with $1\leq \ell\leq n$, we have
\begin{description}
\item[(a)] $\displaystyle 
\mu(L\cap \mathbb{S}^n)\leq  \frac{\ell}{n+1}\cdot \mu(\mathbb{S}^n)$;
\item[(b)] equality  in (a) for a  linear $\ell$-subspace $L\subset \R^{n+1}$ 
with $1\leq d\leq n$ implies the existence of a complementary linear $(n+1-\ell)$-subspace $\widetilde{L}\subset \R^{n+1}$
 such that
${\rm supp}\,\mu\subset L\cup \widetilde{L}$.
\end{description} 

\item[(iii)] If $\frac1{n+2}<\alpha<1$ and $d\mu=f\,d\theta$ for non-negative 
$f\in L^{\frac{n+1}{n+2-\frac1\alpha}}( \mathbb{S}^n)$ with $\int_{\mathbb{S}^n}f>0$.

\end{description}

\end{theorem}

\begin{proof} Let $\alpha>\frac1{n+2}$. After rescaling, we may assume that the $\mu$ in 
\eqref{eq:Alexandrov} is a probability measure.
We consider the 
 seguence $d\mu_k=\frac1{\omega_n}f_k\,d\theta$ of Lemma~\ref{weak-approximation}
of Borel probablity measures whose weak limit is $\mu$ and
 $f_k\in C^\infty( \mathbb{S}^n)$ satisfies $f_k>0$. For each $f_k$, let
$\Omega_k\subset\R^{n+1}$ be the convex body  with $o\in\Omega_k$ 
provided by Theorem~\ref{MongeAmpere-with-entropy}
whose support function $u_k$ is the solution
of the Monge-Ampere equation
\begin{equation}\label{eq:MongeAmpere1k}
u_k^{\frac1\alpha}\,dS_{\Omega_k}=f _k\,d\theta;
\end{equation}
 $\exists \lambda_k>0$ under control, with $|\lambda_k\Omega|=|B(1)|$, $\Omega_k$ satisfies that
\begin{equation}\label{eq:entropy-for-MongeAmperek}
\mathcal{E}_{\alpha, f_k} (\lambda_k\Omega_k)\leq \mathcal{E}_{\alpha, f_k} (B(1)).
\end{equation}
We also need the observations that 
\begin{equation}
\label{Omegakvol}
|\Omega_k|=\frac1{n+1}\int_{\mathbb{S}^n}u_k\,dS_{\Omega_k},
\end{equation}
and if $p=1-\frac1\alpha$, then
\begin{equation}
\label{SOmegakp}
S_{\Omega_k,p}(\mathbb{S}^n)=\int_{\mathbb{S}^n}u_k^{1-\frac1\alpha}\,dS_{\Omega_k}
=\omega_n\oint_{\mathbb{S}^n}f_k=\omega_n.
\end{equation}

We claim that if there exists $\Delta>0$ depending on 
$n$, $\alpha$ and $\mu$ such that
\begin{equation}
\label{Omegakbounded}
{\rm diam}\Omega_k\leq \Delta, 
\mbox{ \ then Theorem~\ref{mainT10} holds.}
\end{equation}
 To prove this claim, we note that
\eqref{Omegakbounded} yields the existence of a subsequence of $\{\Omega_k\}$ tending to a compact convex set 
$\Omega$ with $o\in \Omega$,
which is a convex body by \eqref{SOmegakp} and 
 Lemma~\ref{non-degenerate-weak-limit}. Moreover,  Lemma~\ref{non-degenerate-weak-limit}
also yields that $\Omega$ is an Alexandrov solution of \eqref{eq:Alexandrov}, verifying the claim \eqref{Omegakbounded}.

We divide the rest of the argument verifying Theorem~\ref{mainT10} into   three cases.\\

\noindent{\bf Case 1} {\it $\alpha>1$}

Since $\mu$ is not concentrated to any great subsphere, there exist $\delta\in(0,\frac12)$ depending on $\mu$ such that
$\mu\left(\Psi(z^\bot\cap \mathbb{S}^n,2\delta)\right)\leq 1-2\delta$ for  any $z\in S^{n-1}$. 
It follows from Lemma~\ref{weak-approximation} that we may assume that
\begin{equation}
\label{eq:f-diamestsi000}
\oint_{\Psi(z^\bot\cap \mathbb{S}^n,\delta)}f_k\leq 1-\delta \mbox{ \ for  any $z\in S^{n-1}$}.
\end{equation}
Now Theorem~\ref{MongeAmpere-with-entropy} implies that $\lambda_k\geq c$ for a constant $c>0$ depending on 
$n$, $\delta$ and  $\alpha$, and in turn Theorem~\ref{MongeAmpere-with-entropy}, \eqref{Bentropyalphauppbound}
 and $\frac1{\alpha}-1<0$ yield that
$$
\mathcal{E}_{\alpha, f} (\Omega_k)=\frac{\alpha}{\alpha-1}\cdot\log\lambda_k^{\frac1{\alpha}-1}
+\mathcal{E}_{\alpha, f} (\lambda_k\Omega_k)\leq 
\frac{\alpha}{\alpha-1}\cdot\log\lambda_k^{\frac1{\alpha}-1}
+\mathcal{E}_{\alpha, f} (B(1))\leq C
$$
for a constant $C>0$ depending on 
$n$, $\delta$ and $\alpha$. 
Therefore Theorem~\ref{f-diamests} and \eqref{eq:f-diamestsi000} imply 
that the sequence $\{\Omega_k\}$ is bounded, and in turn
the claim \eqref{Omegakbounded} implies 
 Theorem~\ref{mainT10} if $\alpha>1$.\\

\noindent{\bf Case 2} {\it $\alpha=1$}

The argument is by induction on $n\geq 0$ where we do not put any restriction on the probability meaure $\mu$
in the case $n=0$. For the case $n=0$, we observe that  any finite measure $\mu$ on $S^0$ can represented
in the form $d\mu=u\,dS_\Omega$ for a suitable segment $\Omega\subset\R^1$.

For the case $n\geq 1$, assuming that we have verified Theorem~\ref{mainT10} (ii) in smaller dimensions, we
consider a  Borel measure probability $\mu$ on $S^n$ satisfying (a) and (b). \\

\noindent{\bf Case 2.1} {\it There exists
a  linear $\ell$-subspace $L\subset \R^{n+1}$
with $1\leq \ell\leq n$ and
$\mu(L\cap \mathbb{S}^n)= \frac{\ell}{n+1}\cdot \mu(\mathbb{S}^n)$.}

Let
$\widetilde{L}\subset \R^{n+1}$ be the complementary linear $(n+1-\ell)$-subspace
with ${\rm supp}\,\mu\subset L\cup \widetilde{L}$, and hence 
$\mu(\widetilde{L}\cap \mathbb{S}^n)= \frac{n+1-\ell}{n+1}\cdot \mu(\mathbb{S}^n)$.
It follows by induction that there exist an $\ell$-dimensional compact convex set $K'\subset L$
and an $(n+1-\ell)$-dimensional compact convex set $\widetilde{K}'\subset \widetilde{L}$
such that  $\mu\mres(L\cap S^n)=\ell\,V_{K'}$
and $\mu\mres(\widetilde{L}\cap S^n)=(n+1-\ell)V_{\widetilde{K}'}$.
Finally, for $K=\widetilde{L}^\bot\cap (K'+L^\bot)$ and
$\widetilde{K}=L^\bot\cap (\widetilde{K}'+\widetilde{L}^\bot)$, there exist $\alpha,\tilde{\alpha}>0$ such that
$$
\mu=(n+1)V_{\alpha K+\tilde{\alpha}\widetilde{K}}.
$$

\noindent{\bf Case 2.2} {\it $\mu(L\cap \mathbb{S}^n)< \frac{\ell}{n+1}\cdot \mu(\mathbb{S}^n)$ for 
any  linear $\ell$-subspace $L\subset \R^{n+1}$
with $1\leq \ell\leq n$.}

It follows by a compactness argument that there exists $\delta\in(0,\frac12)$ depending on $\mu$ such that
$\mu(\Psi(L\cap \mathbb{S}^n,2\delta))<(1-2\delta)\cdot \frac{\ell}{n+1}$
for any linear $\ell$-subspace $L$ of $\R^{n+1}$, $\ell=1,\ldots,n$.
We consider the 
 seguence of probability measures
$d\mu_k=\frac1{\omega_n}f_k\,d\theta$ of Lemma~\ref{weak-approximation} tending weakly to $\mu$
such that $f_k>0$, $f_k\in C^\infty(\mathbb{S}^n)$ and 
\begin{equation}
\label{eq:strict-subspace-function0approx}
\mu_k\left(\Psi\left(L\cap \mathbb{S}^n,\delta\right)\right)<(1-\delta)\cdot \frac{\ell}{n+1}
\end{equation}
for any linear $\ell$-subspace $L$ of $\R^{n+1}$, $\ell=1,\ldots,n$.

For each $f_k$, let $\Omega_k\subset\R^{n+1}$ with $o\in\Omega_k$ be the  convex body
provided by Theorem~\ref{MongeAmpere-with-entropy}
whose support function $u_k$ is the solution
of the Monge-Ampere equation
\eqref{eq:MongeAmpere1} and satisfies \eqref{eq:entropy-for-MongeAmpere} with $f=f_k$ and $\lambda=\lambda_k$
where $|B(1)|=|\lambda_k\Omega_k|$ for $\lambda_k>0$, and
\begin{align*}
|\Omega_k|&=\frac1{n+1}\,\int_{\mathbb{S}^n}u_k\det(\bar{\nabla}^2_{ij} u_k+u_k\bar{g}_{ij})\,d\theta=
\frac{\omega_n}{n+1}\,\oint_{\mathbb{S}^n}u_k\det(\bar{\nabla}^2_{ij} u_k+u_k\bar{g}_{ij})\\
&=|B(1)|\,\oint_{\mathbb{S}^n}f_k=|B(1)|,
\end{align*} 
and hence $\lambda_k=1$. In particular, \eqref{Bentropyalphauppbound} yields
$$
\mathcal{E}_{1, f_k} (\lambda_k\Omega_k)\leq \mathcal{E}_{1, f_k} (B(1))\leq \log 2.
$$
Since $\mathcal{E}_{1, f_k} (\Omega_k)$ is bounded,
\eqref{eq:strict-subspace-function0approx} and Theorem~\ref{f-diamests} imply that the 
sequence $\Omega_k$ stays bounded, as well. 
Therefore the claim \eqref{Omegakbounded} yields
 Theorem~\ref{mainT10} if $\alpha=1$.\\

\noindent{\bf Case 3} {\it $\frac1{n+2}<\alpha<1$}

We set  $p=1-\frac1\alpha\in(-n-1,0)$ and $r=\frac{n+1}{n+1+p}>1$, and 
\begin{equation}
\label{eq:taufinalcond}
\tau= \frac12\cdot 2^{-\frac{|p|(n+1)}{|p|+n}},
\end{equation}
and choose $\delta\in(0,\frac12)$ such that
$$
\oint_{\Psi(z^\bot\cap \mathbb{S}^n,2\delta)}f^r\leq \frac{\tau^r}{2^r} 
$$
for  any $z\in S^{n-1}$.
We deduce from Lemma~\ref{weak-approximation} that
if  $z\in S^{n-1}$, then
\begin{equation}
\label{eq:f-diamestsiiifk}
\oint_{\Psi(z^\bot\cap \mathbb{S}^n,\delta)}f_k^r\leq 
\tau^r. 
\end{equation}

We deduce from  \eqref{eq:MongeAmpere1k}, \eqref{Omegakvol} and
$|\lambda_k\Omega_k|=|B(1)|=\frac{\omega_n}{n+1}$ that
\begin{equation}
\label{entropyint}
\oint_{\mathbb{S}^n}u_k^{p}f_k=\frac{n+1}{\omega_n}
\int_{\mathbb{S}^n}u_k\,dS_{\Omega_k}=
\frac{n+1}{\omega_n}\,|\Omega_k|=\lambda_k^{-n-1}.
\end{equation}
In particular, \eqref{Bentropyalphauppbound} and
the upper bound on the entropy yields that
\begin{align}
\nonumber
2^{p}&\leq
\exp\left(p\cdot
\mathcal{E}_{\alpha, f_k} (B(1))\right)
\leq 
\exp\left(p\cdot
\mathcal{E}_{\alpha, f} (\lambda_k\Omega_k)\right)
\leq
\oint_{\mathbb{S}^n}(\lambda_ku_k)^{p}f_k\\
\label{eq:ukfkalphalambdavoliii}
&=\lambda_k^{p}\int_{\mathbb{S}^n}u_k\,dS_{\Omega_k}=
\lambda_k^{p-n}\cdot\frac{n+1}{\omega_n}\cdot|\lambda_k\Omega_k|=\lambda_k^{p-n}.
\end{align}
It follows from \eqref{eq:ukfkalphalambdavoliii} that $\lambda_k\leq 2^{\frac{|p|}{|p|+n}}$, and in turn
 \eqref{entropyint}  yields that
$$
\oint_{\mathbb{S}^n}u_k^{p}f_k\geq 2^{-\frac{|p|(n+1)}{|p|+n}}.
$$
Therefore, $\tau\leq \frac12\oint_{\mathbb{S}^n}u_k^{p}f_k$ 
({\it cf.} \eqref{eq:taufinalcond}), \eqref{eq:f-diamestsiiifk} and
Theorem~\ref{f-diamests} yield
that the sequence $\{\Omega_k\}$ is bounded, and in turn
the claim \eqref{Omegakbounded} implies 
 Theorem~\ref{mainT10} if $\frac1{n+2}<\alpha<1$.
\end{proof}




\bibliographystyle{amsalpha}

\end{document}